\documentclass[12pt,a4paper]{article}
\usepackage[utf8]{inputenc}
\usepackage{tikz}
\usetikzlibrary{arrows,automata}
\tikzset{every state/.style={minimum size=0pt}}
\usepackage{smartdiagram}  
\usepackage{algorithm}
\usepackage{algcompatible}
\usepackage{graphicx}
\usepackage{xcolor}
\usepackage[algo2e,linesnumbered,ruled,vlined]{algorithm2e}
\usepackage[english]{babel}
\usepackage{subfig}
\usepackage{animate}
\usepackage{authblk}
\usepackage[margin=1in]{geometry}
\usepackage[T1]{fontenc}
\usepackage{amssymb, amsthm}
\usepackage{amsmath}
\usepackage{mathtools}
\usepackage{relsize}
\usepackage{todonotes}
\usepackage{hyperref}

\newtheorem{theorem}{Theorem}
\newtheorem{lemma}{Lemma}

\newtheorem{conjecture}{Conjecture}

\newtheorem{observation}{Observation}

\newtheorem*{remark}{Remark}

\begin{document}
\title{\textbf{Towards a conjecture on long induced rainbow paths in  $\Delta$-free graphs}}
\author{N.R. Aravind, Shiwali Gupta, and Rogers Mathew}
\affil{Department of Computer Science and Engineering, Indian Institute of Technology Hyderabad. 
\authorcr\{aravind, cs21resch11002, rogers\}@iith.ac.in}
\date{}
\maketitle

\begin{abstract}
Given a triangle-free graph $G$ with chromatic number $k$ and a proper vertex coloring $\phi$ of $G$, it is conjectured that $G$ contains an induced rainbow path on $k$ vertices under $\phi$. Scott and Seymour proved the existence of an induced rainbow path on $(\log \log \log k)^{\frac{1}{3}- o(1)}$ vertices. We improve this to $(\log k)^{\frac{1}{2}- o(1)}$ vertices. Further, we prove the existence of an induced path that sees $\frac{k}{2}$ colors.
\end{abstract}

\section{Introduction}
We shall use $[n]$ to denote $\{1, 2, \dots, n\}$, where $n$ is a positive integer. All logarithms are to the base $2$ in this paper.
We consider finite and simple graphs throughout the paper. 
Given a graph $G$, we use $V(G)$ to denote its vertex set and $E(G)$ to denote its edge set. For a vertex $v \in V(G)$, we shall use $N_G(v)$ to denote its neighborhood in $G$. 
A \emph{proper vertex coloring} of $G$ using $k$ colors (or a $k$-coloring of $G$) is a function $\phi ~:~ V(G) \longrightarrow [k]$ such that for every edge $uv \in E(G)$, $\phi(u) \neq \phi(v)$. 
The minimum positive integer $k$ for which there exists a $k$-coloring of $G$ is called the chromatic number of $G$, denoted by $\chi(G)$. 
We shall use the term \emph{colored graph} to denote a tuple $(G, \phi)$, where $G$ is the graph under consideration and $\phi$ is a proper coloring of $G$.

Given a colored graph $(G, \phi)$,
a path $P$ in $G$ is called \emph{rainbow} if all the vertices of $P$ have different colors under $\phi$.
A path $P$ is called \emph{$k$-colorful} if the set of colors appearing on the vertices of $P$ has size at least $k$. Equivalently, $P$ contains vertices of $k$ distinct colors, but the number of vertices in $P$ can be much larger as compared to $k$.
A path $P$ in $G$ is called \emph{induced} if no two nonconsecutive vertices of $P$ are adjacent in $G$. The \emph{order} of a path is the number of vertices in it.

Given a triangle-free graph $G$, Gy\'arf\'as \cite{gyarfas1987problems} proved (by induction on $\chi(G)$) that $G$ contains an induced path on $\chi(G)$ vertices. 
The Gallai--Roy Theorem \cite{gallai1968directed, roy1967nombre} implies that every oriented graph contains a directed path of order equal to its chromatic number: by orienting each edge toward the endpoint with the larger color, every directed path becomes rainbow, and hence a rainbow path of order equal to its chromatic number exists. 
To summarize, given a colored triangle-free graph $(G, \phi)$, (i) $G$ contains an induced path of order $\chi(G)$ , and (ii) $G$ contains a rainbow path of order $\chi(G)$. The following conjecture on induced rainbow paths was posed by the first author (see \cite{babu2019induced}) of this paper.

\begin{conjecture} \cite{babu2019induced}
Every colored triangle-free graph $(G, \phi)$ contains an induced rainbow path on $\chi(G)$ vertices.
\end{conjecture}

While this conjecture remains open in general, several partial results are known. Babu et. al. \cite{babu2019induced} studied a weaker version of the conjecture and proved that it holds when the girth of the graph equals its chromatic number. The \emph{girth} of a graph is the length of its shortest cycle; if no cycle exists, the girth is infinite. 
Subsequently, Gy\'arf\'as and S\'ark\"ozy \cite{gyarfas2016induced} investigated the conjecture on colored graphs $(G, \phi)$ having girth at least five. They proved that $(G,\phi)$ contains an induced rainbow path on $(\log \chi(G))^{1 - o(1)}$ vertices. Very recently, this result on graphs of girth at least 5 was improved by Basavaraju et al. \cite{basavaraju2025variants} who showed that  $(G, \phi)$ contains an induced rainbow path on $\Omega(\sqrt{\chi(G)})$ vertices. The proofs of both the aforementioned results crucially use the fact $G$ is $C_4$-free and hence cannot be extended to triangle-free graphs. The conjecture in its full generality, without any girth restriction, was addressed by Scott and Seymour \cite{DBLP:journals/combinatorics/ScottS17}. They showed that every colored triangle-free graph $(G, \phi)$ contains an induced rainbow path on $(\log\log\log \chi(G))^{\frac{1}{3}-o(1)}$ vertices.

In this work, we improve the bound given by Scott and Seymour in \cite{DBLP:journals/combinatorics/ScottS17}. We show that every colored triangle-free graph $(G, \phi)$ contains an induced rainbow path on $(\log \chi(G))^{\frac{1}{2}- o(1)}$ vertices.
Further, if we relax the requirement from a rainbow path to a colorful path, then we show that a colored triangle-free graph $(G, \phi)$ contains an induced $\frac{\chi(G)}{2}$-colorful path.

\section{Our results}
We begin this section by defining a few definitions and notations that will be used in the proof. The following definition is from \cite{DBLP:journals/combinatorics/ScottS17}.
A \emph{grading} of a graph $G$ is a sequence $(W_1, W_2, \dots, W_n)$ of pairwise disjoint subsets of $V(G)$ such that $W_1 \uplus W_2 \uplus \dots \uplus W_n = V(G)$. If $\chi(G[W_i]) \le k$ for every $i \in [n]$, then $(W_1, W_2, \dots, W_n)$ is called a \emph{$k$-colorable grading}. Given a grading $(W_1, W_2, \dots, W_n)$ and vertices $u,v \in V(G)$, we say that $u$ is \emph{later than} $v$ if $u \in W_i$, $v \in W_j$, and $i>j$.

Given an undirected graph $G$ (resp., oriented graph $\overrightarrow{G}$), and a subset $S  \subseteq V(G)$ (resp., $V(\overrightarrow{G}))$, we use $G[S]$ (resp., $\overrightarrow{G}[S]$) to denote the subgraph of $G$ (resp., $\overrightarrow{G}$) induced by $S$. 
We shall use $P= v_1$ - $v_2$ - $\cdots$ - $v_k$ to denote a path of order $k$, and $\overrightarrow{P} = v_1 \xrightarrow{} v_2 \xrightarrow{} \cdots \xrightarrow{} v_k$ to denote a directed path of order $k$.

Let $P= v_1$ - $v_2$ - $\cdots$ - $v_k$ be a $k$-vertex path. For any $1 \le i \le k$, we shall use $P_{v_i}$ to denote the path $v_1$ - $v_2$ - $\cdots$ - $v_i$. 
Let $P'= v_1$ - $v_2$ - $\cdots$ - $v_\ell$ and $P'' = w_1$ - $w_2$ - $\cdots$ - $w_r$ be two paths in a graph $G$. Suppose $v_{\ell}w_1 \in E(G)$. Then, we use $P'$ - $P''$ to denote the path $v_1$ - $v_2$ - $\cdots$ - $v_\ell$ - $w_1$ - $w_2$ - $\cdots$ - $w_r$.

The lemma stated below is a reformulation of Statement~2.2 in \cite{DBLP:journals/combinatorics/ScottS17}. 
The main difference lies in the choice of the parameter $r$. 
In \cite{DBLP:journals/combinatorics/ScottS17}, the value of $r$ is derived from the theorem of Galvin, Rival, and Sands \cite{galvin1982ramsey}, where $r = 2^{2^{2^{s^{3}}}}$. 
In contrast, we use $r = 4 \cdot 2^{2(s-1) \log (s-1)}$, which is significantly smaller. 
This improvement leads to a better bound on the order of a longest induced rainbow path in a triangle-free graph.

\begin{lemma}
\label{lem k_{1,s}}
Let $s \ge 3$ be an integer and let $r=4 \cdot 2^{2(s-1)\log (s-1)}$. Let $(G, \phi)$ be a colored graph with $\chi(G) \ge k\cdot r$, where $k$ is some positive integer. Let $(W_1, W_2, \dots, W_n)$ be a $k$-colorable grading of $G$. Then, one of the following statements always holds:\\
(i) $G$ contains an induced rainbow path on $s$ vertices under $\phi$.
(ii) $\exists i \in [n]$ and a vertex $v \in W_i$, and a set of $s$ vertices, pairwise with different colors under $\phi$, all later than $v$, and all adjacent to $v$.
\end{lemma}

\begin{proof}
Since $(W_1, W_2, \dots, W_n)$ is a $k$-colorable grading of $G$, each $W_i$, $i \in [n]$, admits a proper coloring with $k$ colors, and hence can be partitioned into $k$ parts (some of these parts could be empty sets), where each part is an independent set in $G$.
Let $W_{i_1}, W_{i_2}, \dots, W_{i_k}$ denote the $k$ parts of $W_i$, $i \in [n]$. 
We now obtain a partition of $V(G)$ into $k$ parts, denoted by $(Z_1, Z_2, \dots, Z_k)$, where each $Z_j := W_{1_j} \uplus W_{2_j} \uplus \dots \uplus W_{n_j}$. 
Since  $\chi(G) \ge k\cdot r$, there exists some $Z_j$, $1 \le j \le k$, with $\chi(G[Z_j]) \ge r$. 
Let $G_j$ denote $G[Z_j]$.
We now orient each edge of $G_j$ from its lower-colored vertex to its higher-colored vertex.
The resulting oriented graph is denoted by $\overrightarrow{G_j}$.
The observations below follow from the definition of $\overrightarrow{G_j}$.

\begin{observation}
\label{obs acyclic}
$\overrightarrow{G_j}$ is a directed acyclic graph.
\end{observation}

\begin{observation}
\label{obs rainbow}
Every directed path in $\overrightarrow{G_j}$ is a rainbow path.
\end{observation}
     
We define an ordering $\pi$ on $V(\overrightarrow{G_j}) = W_{1_j} \uplus W_{2_j} \uplus \dots \uplus W_{n_j}$ such that all the vertices of $W_{1_j}$ appear first from left to right (with ties broken arbitrarily), followed by all the vertices of $W_{2_j}$, and so on.
We split the oriented graph $\overrightarrow{G_j}$ into two subgraphs, $\overrightarrow{G_j^1}$ and $\overrightarrow{G_j^2}$, each defined on the full vertex set $V(\overrightarrow{G_j})$. 
The edge set of a subgraph $\overrightarrow{G_j^1}$ is the set of all the forward edges (edges that are going from left to right) of $\overrightarrow{G_j}$ and the edge set of a subgraph $\overrightarrow{G_j^2}$ is the set of all the backward edges (edges that are going from right to left) of $\overrightarrow{G_j}$ with respect to $\pi$. From Observations \ref{obs acyclic} and \ref{obs rainbow}, we get the following.

\begin{observation}
\label{obs path_in_G}
A directed induced path in $\overrightarrow{G_j^i}$, $i \in \{1, 2\}$, is a directed induced rainbow path in $\overrightarrow{G_j}$. Further, it is an induced rainbow path in $G_j$ as well as in $G$. 
\end{observation}

By the Gallai--Vitaver--Roy Theorem (\cite{lovasz2007combinatorial}, Ex. 9.9), we have $ r \le \chi(\overrightarrow{G_j}) = \chi(\overrightarrow{G_j^1} \cup \overrightarrow{G_j^2}) \le \chi(\overrightarrow{G_j^1}) \cdot \chi(\overrightarrow{G_j^2})$. 
Thus, $\chi(\overrightarrow{G_j^i}) \ge \sqrt{r}$ for some $i \in \{1, 2\}$.
Let $\overrightarrow{P^1} := x_1 \xrightarrow{} x_2 \xrightarrow{} \cdots \xrightarrow{} x_{\ell_1}$ and $\overrightarrow{P^2} := y_1 \xrightarrow{} y_2 \xrightarrow{} \cdots \xrightarrow{} y_{\ell_2}$  be longest directed paths in $\overrightarrow{G_j^1}$ and $\overrightarrow{G_j^2}$, respectively.
By the Gallai--Roy Theorem \cite{gallai1968directed, roy1967nombre} and Observation \ref{obs rainbow}, we have 

\begin{observation}
\label{obs path_in_p1_or_in_p2}
Both $\overrightarrow{P^1}$ and $\overrightarrow{P^2}$ are rainbow paths with $\ell_1 \ge \sqrt{r}$ or $\ell_2 \ge \sqrt{r}$.
\end{observation}

We make the following observation based on our construction of $\overrightarrow{G_j^1}$ and $\overrightarrow{G_j^2}$ from $\overrightarrow{G_j}$.

\begin{observation}
\label{obs later}
(i) Let $1 \le a < b \le \ell_1$. Then, $x_b$ is later than $x_a$ with respect to the grading $(W_1, W_2, \dots, W_n)$.\\
(ii) Let $1 \le a < b \le \ell_2$. Then, $y_a$ is later than $y_b$ with respect to the grading $(W_1, W_2, \dots, W_n)$.
\end{observation}

If Statement (ii) of the lemma holds, then our proof is complete. Suppose Statement (ii) does not hold. Then from Observations \ref{obs path_in_p1_or_in_p2} and \ref{obs later}, we have the following. In what follows, let $X := V(\overrightarrow{P^1})= \{x_1, x_2, \dots, x_{\ell_1}\}$ and $Y := V(\overrightarrow{P^2})= \{y_1, y_2, \dots, y_{\ell_2}\}$.

\begin{observation}
\label{obs s-1 neighbors}
(i) Every vertex $x_i$ in $\overrightarrow{P^1}$ has at most $s-1$ out-neighbors in $\overrightarrow{G_j^1}[X]$.\\
(ii) Every vertex $y_i$ in $\overrightarrow{P^2}$ has at most $s-1$ in-neighbors in $\overrightarrow{G_j^2}[Y]$.
\end{observation}

Suppose $\ell_1 \ge \sqrt{r}$. Consider the oriented graph $\overrightarrow{G_j^1}[X]$. Apply a Breadth-First Search (BFS) on $\overrightarrow{G_j^1}[X]$ by taking $x_1$ as the start vertex and by using only out-edges for the traversal.
The BFS tree thus obtained is rooted at $x_1$ and every vertex $x_i$ in the tree has at most $s-1$ child nodes due to Observation \ref{obs s-1 neighbors}(i). 
If $d$ denotes the depth of this tree on $\sqrt{r}$  vertices, then $\sqrt{r} \le (s-1)^0 + (s-1)^1 + \dots + (s-1)^d \le 2(s-1)^d$ (since $s \ge 3$). 
Thus, $d \ge \frac{\log \big(\frac{\sqrt{r}}{2}\big)}{\log (s-1)} = \frac{\log \big(2^{(s-1) \log (s-1)}\big) }{\log (s-1)} = s-1$.
Thus, the depth of the BFS tree is at least $s-1$ and therefore $\overrightarrow{G_j^1}[X]$ has an directed induced path on $s$ vertices. It is easy to see that such a path is also an directed induced path in $\overrightarrow{G_j^1}$. Using Observation \ref{obs path_in_G}, we can conclude that such a path is an induced rainbow path in $G$.

Suppose $\ell_2 \ge \sqrt{r}$. Consider the oriented graph $\overrightarrow{G_j^2}[Y]$. Apply a BFS on $\overrightarrow{G_j^2}[Y]$ by taking $y_{\ell_2}$ as the start vertex and by using only in-edges for the traversal. We omit the rest of the proof as the arguments are analogous to the ones used in the case $\ell_1 \ge \sqrt{r}$ explained in the above paragraph.
\end{proof}

Let  $r$ be as defined in the statement of Lemma \ref{lem k_{1,s}}. Let $w_s =0$, and for $j= s-1, \dots, 1$, let $w_j = w_{j +1}  r + 1$. Let $c$ be defined as 
\begin{equation}
\label{eq c_value}
c = (w_1 +1)r    
\end{equation}
The lemma stated below is a reformulation of Statement 2.3 in \cite{DBLP:journals/combinatorics/ScottS17}.

\begin{lemma} \cite{DBLP:journals/combinatorics/ScottS17}
\label{lem seymour main} 
Let $s$ be any positive integer. Let $r=4 \cdot 2^{2(s-1)\log (s-1)}$ as defined in Lemma \ref{lem k_{1,s}}. Let $c$ be as defined in Equation \ref{eq c_value}. Then, every colored triangle-free graph $(G, \phi)$ with $\chi(G) > c$ contains some induced rainbow path on $s$ vertices.
\end{lemma}

\begin{remark}
The statement of Lemma \ref{lem seymour main} can be generalized to graphs of bounded clique number as shown in Statement 2.3 by Scott and Seymour in \cite{DBLP:journals/combinatorics/ScottS17}.
\end{remark}

From Lemmas \ref{lem k_{1,s}} and \ref{lem seymour main}, we have the following theorem. 

\begin{theorem}
    Every colored triangle-free graph $(G, \phi)$ contains an induced rainbow path on $(\log \chi (G))^{\frac{1}{2}- o(1)}$ vertices under $\phi$.
\end{theorem}

We now relax the requirement from obtaining a rainbow path to obtaining a colorful path. Consider a colored triangle-free graph $(G, \phi)$. What is the largest $k$ for which $G$ contains a $k$-colorful path? If $\phi$ is a coloring that uses only $\chi(G)$ colors, then $k \le \chi(G)$. Thus, the largest $k$ one can hope for is $k = \chi(G)$.

\begin{theorem}
    Let $(G, \phi)$ be a colored triangle-free graph. Under $\phi$, for every $v \in V(G)$, there exists an induced $\frac{\chi (G)}{2}$-colorful path starting at $v$.
\end{theorem}

\begin{proof}
    We shall prove this by induction on $\chi(G)$. It is easy to see that the statement is true for the base case when $\chi(G) = 1$.
    Assume that the statement is true for $\chi(G) < k$, where $k \ge 2$.
    Let $G$ be a minimal graph with $\chi (G) = k$ for which the statement is false. 
    Let $v$ be a vertex in $G$ with $\phi(v) = c$.
    For the induction step, we first remove all the vertices that are colored with color `c' from $G$, including the vertex $v$. LLet the resulting graph be $G'$; we note that $\chi(G')=k-1$.
    Thus, there exists at least one component in $G'$ whose chromatic number is $k-1$. 
    
    Now consider any one such component $C_1$ of $G'$. Let $P$ be a shortest path from $v$ to $C_1$ in $G$ with $|V(P) \cap V(C_1)| = 1$. 
    The first vertex of $P$ is $v$, the second vertex of $P$ is the neighbor of $v$ in $P$, and so on. Let $w$ be the penultimate (or the second last) vertex of $P$.
    Note that $w$ is not in $C_1$ and $w =v$ if $v$ has a neighbor in $V(C_1)$ in $G$.
    Let $N_G(w) \cap V(C_1) = \{w_1, w_2, \dots, w_{\ell}\}$.  Since $G$ is triangle-free, $\{w_1, w_2, \dots, w_{\ell}\}$ forms an independent set. 
    Further, since $P$ is a shortest path from $w$ to $C_1$, the only neighbor of $w_i$, $1 \le i \le \ell$, in $P$ is the vertex $w$. 
    Let $G''$ be the subgraph of $C_1$ obtained by deleting $w_1, w_2, \dots, w_{\ell}$ from $C_1$. 
    Since the chromatic number of $C_1$ is $k-1$, there exists a component in $G''$, say $C_2$, whose chromatic number is at least $k-2$.

    Let  $w_i \in \{w_1, w_2, \dots, w_{\ell}\}$ be a vertex that had a neighbor in $V(C_2)$ in the component $C_1$. Let $G'''$ be the connected subgraph of $C_1$ induced by $V(C_2) \cup \{w_i\}$. We have $\chi(G''') \ge k-2$ and by induction hypothesis, there is an induced $\frac{k-2}{2}$-colorful path $Q$ in $G'''$ starting at $w_i$, under the coloring $\phi$.

    It follows from our construction that the path $Q$ does not have any vertex with color `c'.
    Now consider the path $R= P_w$ - $Q$. This path is well defined as $P_w$ is a path in $G$ between $v$ and $w$, $Q$ is a path in $G$ starting at vertex $w_i$, $ww_i \in E(G)$, and $V(P_w) \cap V(Q) = \emptyset$.
    Since the path $R$ has vertex $v$ in it which sees color `c', the path $R$ sees $1 + \frac{k-2}{2}  = \frac{k}{2}$ colors.
    
    What remains to be shown is that $R$ is an induced path. By definition, both paths $P_w$ and $Q$ are induced. Therefore, it suffices to show that there is no edge between a vertex of $P_w$ and a vertex of $Q$, except the edge $ww_i$. 
    Since $P$ is a shortest path from $v$ to $C_1$, no vertex of $P_w$, other than the vertex $w$, has a neighbor in $Q$; otherwise, such an edge would yield a shorter path from $v$ to $C_1$, contradicting the choice of $P$. From our construction of $Q$, the only neighbor of $w$ in $Q$ is the vertex $w_i$. Therefore the concatenation of $P_w$ and $Q$ forms an induced path $R$.

\end{proof}

\bibliographystyle{plain}
\bibliography{Thesis_reference}
\end{document}